\documentclass[oneside,english]{amsart}
\usepackage[T1]{fontenc}
\usepackage{geometry}
\usepackage{float}
\usepackage{mathtools}
\usepackage{amstext}
\usepackage{amsthm}
\usepackage{amssymb}
\usepackage{graphicx}
\makeatletter
\theoremstyle{plain}
\newtheorem{thm}{\protect\theoremname}
\theoremstyle{remark}
\newtheorem*{rem*}{\protect\remarkname}
\theoremstyle{plain}
\newtheorem{lem}[thm]{\protect\lemmaname}
\theoremstyle{definition}
\newtheorem{example}[thm]{\protect\examplename}

\usepackage{amsfonts}
\usepackage[mathscr]{eucal}

\makeatother

\usepackage{babel}
\providecommand{\examplename}{Example}
\providecommand{\lemmaname}{Lemma}
\providecommand{\remarkname}{Remark}
\providecommand{\theoremname}{Theorem}

\begin{document}
\title{A Generalization of the Erd\H{o}s-Kac Theorem}
\author{Joseph Squillace\\
University of Rhode Island\\
Department of Computer Science $\&$ Statistics}
\begin{abstract}
Given $n\in\mathbb{N}$, let $\omega\left(n\right)$ denote the number
of distinct prime factors of $n$, let $Z$ denote a standard normal
variable, and let $P_{n}$ denote the uniform distribution on $\left\{ 1,\ldots,n\right\} $.
The Erd\H{o}s-Kac Theorem states that $$P_{n}\left(m\le n:\omega\left(m\right)-\log\log n\le x\left(\log\log n\right)^{1/2}\right)\to\mathbb{P}\left(Z\le x\right)$$
as $n\to\infty$; i.e., if $N\left(n\right)$ is a uniformly distributed variable on $\lbrace 1,\ldots,n \rbrace$, then $\omega\left(N\left(n\right)\right)$ is asymptotically normally distributed as $n\to \infty$ with both mean and variance equal to $\log \log n$. The contribution of this paper is a generalization of the Erd\H{o}s-Kac Theorem to a larger class
of random variables by considering perturbations of the uniform probability
mass $\frac{1}{n}$ in the following sense. Denote by $\mathbb{P}_{n}$
a probability distribution on $\left\{ 1,\ldots,n\right\} $ given
by $\mathbb{P}_{n}\left(i\right)=\frac{1}{n}+\varepsilon_{i,n}$.
By providing some constraints on the $\varepsilon_{i,n}$'s, sufficient conditions are stated in order
to conclude that $$\mathbb{P}_{n}\left(m\le n:\omega\left(m\right)-\log\log n\le x\left(\log\log n\right)^{1/2}\right) \to \mathbb{P}\left(Z\le x\right)$$ as $n\to\infty.$

The main result will be applied to prove that the number of distinct
prime factors of a positive integer with either the Harmonic$\left(n\right)$
distribution or the Zipf$\left(n,s\right)$ distribution also tends to the normal distribution $\mathcal{N}\left(\log\log n,\log\log n\right)$ as $n\to\infty$ (and as $s\to1$ in the case of a Zipf variable).
\end{abstract}

\maketitle

\section{Introduction}

Given a natural number $n$, the number of distinct prime factors
of $n$ is denoted $\omega\left(n\right)$. For example, $\omega\left(2^{3}5^{2}7\right)=3$.
The function $\omega$ may be written as $\omega\left(n\right)=\sum_{p\vert n}1$, where the sum is over all prme factors of $n$. In 1917, Hardy and Ramanujan (p. 270 of \cite{key-3}) proved
that the number of distinct prime factors of a natural number $n$
is about $\log\log n$. In particular, they showed that the normal
order of $\omega\left(n\right)$ is $\log\log n-$i.e., for every
$\varepsilon>0$, the proportion of the natural numbers for which
the inequalities 
\[
\left(1-\varepsilon\right)\log\log n\le\omega\left(n\right)\le\left(1+\varepsilon\right)\log\log n
\]
do not hold tends to $0$ as $n\to\infty$. Informally speaking, the
Erd\H{o}s-Kac Theorem generalizes the Hardy-Ramanujan Theorem by showing
that $\omega\left(n\right)$ is about $\log\log n+Z\cdot\sqrt{\log\log n}$,
where $Z\sim\mathcal{N}\left(0,1\right)$. More precisely, the Erd\H{o}s-Kac Theorem is the following result (p. 738 of \cite{key-2}).
\begin{thm}
Let $P_{n}$ denote the uniform distribution\footnote{I.e., if $U$ is uniformly distributed on $\left\{ 1,2,\ldots,n\right\} $,
then for any subset $A\subseteq\left\{ 1,2,\ldots,n\right\}$, $P_{n}\left(A\right)=\mathbb{P}\left(U\in A\right)$.} on $\left\{ 1,2,,\ldots,n\right\} $. As $n\to\infty$, 
\[
P_{n}\left(m\le n:\omega\left(m\right)-\log\log n\le x\left(\log\log n\right)^{1/2}\right)\to\mathbb{P}\left(Z\le x\right).
\]
\end{thm}
\medskip{}
Figure 1 shows plots of the values of $\omega\left(n\right)$ for $n$ between, respectively, 1-100, 1-1000, and 1-10000. 
\begin{figure}[H]
\caption{Some values of $\omega\left(n\right)$ plotted using Mathematica.}
\centering{}\includegraphics[width=5cm,height=5cm,keepaspectratio]{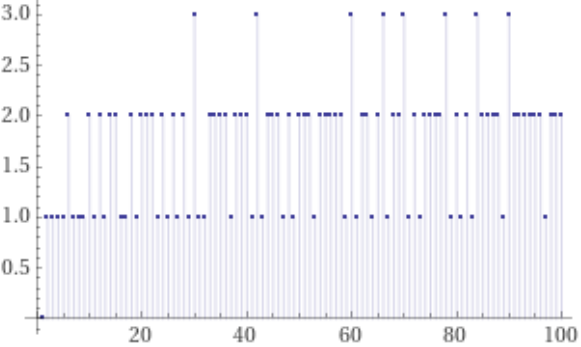}\includegraphics[width=5cm,height=5cm,keepaspectratio]{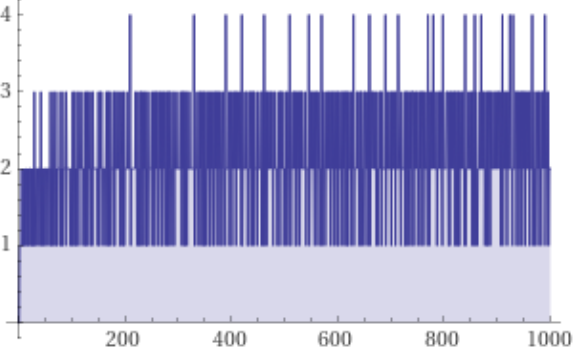}\includegraphics[width=5cm,height=5cm,keepaspectratio]{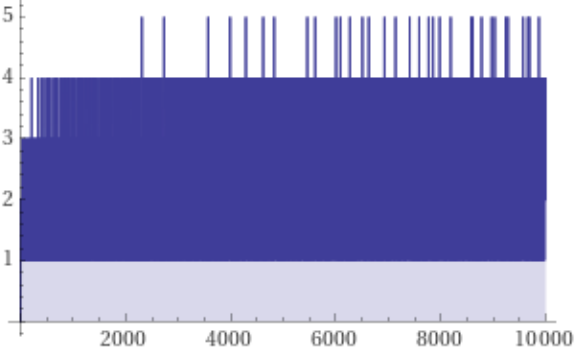}
\end{figure}
\noindent Furthermore, Figure 1 illustrates how slowly the values of $\omega\left(n\right)$ grow along with their variation, and this is consistent with the parameters $\mu=\sigma^2=\log \log n$ appearing in the limiting normal distribution. Theorem 1 suggests that, for example, since $\log\log\left(e^{\left(e^4\right)}\right)=4$,
integers near $e^{\left(e^4\right)}\approx 514843556263457212366848$ have, on average, $4$ distinct
prime factors. Using Mathematica, computing the mean of $\omega\left(n\right)$ for $n$ within 1000000 of $514843556263457212366848$ yields
\[
\text{N[Mean[Table[PrimeNu[n], {n, 514843556263457211366848, 514843556263457213366848}]], 3]}=4.27.
\]
Several generalizations of Theorem 1 exist. For example, Liu \cite{key-4}
extends Theorem 1 to the setting of free abelian modules other than
the positive integers. A recent generalization of Theorem 1 was considered
in \cite{key-7}, where Sun and Wu showed that Theorem 1 corresponds
to a special case of Theorem 1 of \cite{key-7} once the parameter
$l$ is set to $0$. In particular, they considered integrals of
the form $\int_{-\infty}^{z}v^{l}e^{\frac{-v^{2}}{2}}\text{d}v$ and
bounds for sums of the form 
\[
\frac{1}{x}\sum_{n\le x,\frac{\omega\left(n\right)-\log\log n}{\sqrt{\log\log n}}\le z}\left(\frac{\omega\left(n\right)-\log\log x}{\sqrt{\log\log x}}\right)^{l};
\]
upon setting $l=0$, these expressions become the CDF of the standard normal
distribution and the proportion of $n$ satisfying $\frac{\omega\left(n\right)-\log\log n}{\sqrt{\log\log n}}\le z$,
respectively. A generalization given by Saidak \cite{key-6}, assumes a quasi Generalized
Riemann Hypothesis for Dedekind zeta functions and concerns a function
$f_{a}\left(p\right)$ defined to be the minimal $e$ for which $a^{e}\equiv1\mod p$.
Saidak shows that the limit of $$\frac{\left\{ p\le x:A\le\frac{\omega\left(f_{a}\left(p\right)\right)-\log\log p}{\sqrt{\log\log p}}\le B\right\} }{\pi\left(x\right)}$$
as $x\to\infty$ tends to the standard normal CDF $\Phi\left(B\right)-\Phi\left(A\right)$.
There are several generalizations of Theorem 1 to algebraic number
fields, such as \cite{key-5} where Pollack considers the number of
principal ideals. A similarity in the generalizations mentioned is that they consider
limiting distributions of $\omega\left(\cdot\right)$ for uniformly distributed random integers.
The contribution of this paper is to extend the Erd\H{o}s-Kac Theorem to
a larger class of random variables, other than a uniformly distributed
variable, on $\left[n\right]\coloneqq\left\{ 1,2,\ldots,n\right\}$
which also have, asymptotically, $\log\log n+Z\cdot\sqrt{\log\log n}$ many distinct prime factors.

\subsection{The Main Theorem}

\noindent Define a probability distribution $\mathbb{P}_{n}$
on $\left[n\right]$ given by 
\begin{equation}
\mathbb{P}_{n}\left(i\right)=\left(\frac{1}{n}+\varepsilon_{i,n}\right),\label{eq:1}
\end{equation}
and impose the constraint that for each $k$-tuple $\left(p_{1},\ldots,p_{k}\right)$
consisting of distinct primes, 
\begin{equation}
\lim_{n\to\infty}\sum_{l=1}^{\left\lfloor \frac{n}{p_{1}\cdots p_{k}}\right\rfloor }\varepsilon_{lp_{1}\cdots p_{k},n}=0.\label{eq:2}
\end{equation}
The constraint $\left(\ref{eq:2}\right)$ is applied in $\S$2 where
an analogue of Kac's heuristic for $\mathbb{P}_{n}$ is provided, and
the analogue of Kac's heuristic suggests that for a $\mathbb{P}_{n}$-distributed variable $X$,
the events $\left\{p_1\text{ divides }X\right\},\ldots,\left\{p_k\text{ divides }X\right\}$
are independent when $p_1,\ldots p_k$ are distinct primes. Analogous to the
case of the development of Theorem 1, it was the independence of these
events that suggested a Gaussian law of errors. Moreover, it will be shown that $\left(\ref{eq:2}\right)$ implies $\lim_{n\to\infty}\mathbb{P}_{n}\left(p\text{ divides }X\right)=\frac{1}{p}$ (and it is easy to show that $\lim_{n\to\infty}P_{n}\left(p\text{ divides }X\right)=\frac{1}{p}$).

Due to the axioms of probability, the $\varepsilon_{i,n}$'s satisfy
\begin{align}
\sum_{i=1}^{n}\varepsilon_{i,n}= & 0,\label{eq:3}\\
\varepsilon_{i,n}\in & \left[-\frac{1}{n},1-\frac{1}{n}\right].\label{eq:4}
\end{align}
The motivation for defining $\mathbb{P}_{n}$ in terms of the uniform
distribution is due to Durrett's proof (Theorem 3.4.16 in \cite{key-1})
of the Erd\H{o}s-Kac Theorem. Replacing the uniform distribution
$P_{n}$ with the new distribution $\mathbb{P}_{n}$ in Durrett's
proof naturally yields some constraints that the terms $\varepsilon_{i,n},i\le n,$
must satisfy in order to conclude that an integer-valued random variable with the
$\mathbb{P}_{n}$ distribution has about $\log\log n+Z\cdot\sqrt{\log\log n}$
distinct prime factors. Our main result is the following theorem,
where $\left\lfloor \cdot\right\rfloor $ denotes the floor function.
\begin{thm}
Let $Z\sim\mathcal{N}\left(0,1\right)$. Suppose the following statements
are true.

\begin{itemize}
\item There exists a constant $C$ such that for all $n$ and for all primes
$p$ with $p>n^{1/\log\log n}$,
\begin{equation}
\sum_{l=1}^{\left\lfloor \frac{n}{p}\right\rfloor }\varepsilon_{lp,n}\le\frac{C}{p}.\label{eq:5}
\end{equation}
\item There exists a constant $D$ such that 
\begin{equation}
0\le\sum_{l=1}^{\left\lfloor \frac{n}{p_{1}\cdots p_{k}}\right\rfloor }\varepsilon_{lp_{1}\cdots p_{k},n}\le\frac{D}{n}\label{eq:6}
\end{equation}
for all $n$ and, for each $k$, all $k$-tuples $\left(p_{1},\ldots,p_{k}\right)$
consisting of distinct primes of size at most $n^{1/\log\log n}$. 
\end{itemize}
Let $\mathbb{P}_{n}^{*}$ denote a probability distribution obtained
by imposing the constraints $\left(\ref{eq:5}\right)$ and $\left(\ref{eq:6}\right)$
on $\mathbb{P}_{n}$. As $n\to\infty,$ 
\begin{align*}
\mathbb{P}_{n}^{*}\left(m\le n:\omega\left(m\right)-\log\log n\le x\left(\log\log n\right)^{1/2}\right) & \to\mathbb{P}\left(Z\le x\right).
\end{align*}
\end{thm}

\begin{rem*}
If $\varepsilon_{i,n}=0$ for all $i\le n$, then $\mathbb{P}_{n}^{*}=P_{n}$ and Theorem 1 is obtained.
\end{rem*}

\subsection{Outline}

In $\S2$, Kac's heuristic for the number of distinct
prime factors of a uniformly distributed variable on $\mathbb{N}$ is stated,
and then an analogue of the heuristic is provided for any $\mathbb{P}_{n}$-distributed
random variable. Just as Kac's heuristic suggested the appearance
of the normal distribution in Theorem 1 along with the parameters $\mu = \sigma^2 = \log \log n$, the analogue of Kac's heuristic
for $\mathbb{P}_{n}$ is the motivation for the conclusion of Theorem
2. The proof of Theorem 2 is provided in $\S$3; the proof applies the method of moments and is motivated by Durrett's proof of the Erd\H{o}s-Kac Theorem (Theorem 3.4.16 in
\cite{key-1}). Moreover, in $\S$3, the constraints $\left(\ref{eq:5}\right)$ and $\left(\ref{eq:6}\right)$ are applied to ensure that $\mathbb{P}_n$ also satisfies Durrett's method of moment bounds.
In $\S$4, the values $\varepsilon_{i,n},i\le n,$ are given for the Harmonic distribution
on $\left[n\right]$, and it is shown that the $\varepsilon_{i,n}$'s satisfy
constraints $\left(\ref{eq:2}\right)$, $\left(\ref{eq:5}\right)$,
and $\left(\ref{eq:6}\right)$. Therefore, the conclusion of Theorem
2 is true if $\mathbb{P}_{n}^{*}$ is replaced with the harmonic distribution. In a similar fashion, it is shown that a Zipf$\left(n,s\right)$-distribution
can also replace $\mathbb{P}_{n}^{*}$ in Theorem 2, as long as $s\to1$
as $n\to\infty$. I.e., the number of distinct prime factors of an
integer-valued random variable from the Harmonic$\left(n\right)$
distribution distribution tends to $\mathcal{N}\left(\log\log n,\log\log n\right)$
as $n\to\infty$; and the number of distinct prime factors of an integer-valued
random variable from the Zipf$\left(n,s\right)$ distribution tends
to $\mathcal{N}\left(\log\log n,\log\log n\right)$ as $n\to\infty$
and $s\to1$.

\section{Kac's Heuristic and its Analogue for $\mathbb{P}_{n}$}

Kac's heuristic\footnote{The heuristic is based on the connection between independence and
a Gaussian law of errors.} for the uniform distribution (pp. 154-155 of \cite{key-1}) suggests
the statement of Theorem 1 and is based on the fact that given a random
integer $n\in\mathbb{N},$ the events $\left\{ p\text{ divides }n\right\}$
and $\left\{ q\text{ divides }n\right\} $ are independent when $p$
and $q$ are distinct primes.

Now consider the probability distribution $\mathbb{P}_{n}$ and
its behavior in the limit. Given a prime $p$, let $A_{p}$ denote
the set of positive integers divisible by $p$. Letting $\mathbf{1}_{\left\{ \cdot\right\}}$
denote an indicator random variable,
\begin{align*}
\mathbb{P}_{\infty}\left(A_{p}\right) & \coloneqq\lim_{n\to\infty}\mathbb{P}_{n}\left(A_{p}\right)\\
 & =\lim_{n\to\infty}\sum_{l=1}^{\left\lfloor \frac{n}{p}\right\rfloor }\mathbb{P}_{n}\left(lp\right)\\
 & \overset{\left(\ref{eq:1}\right)}{=}\lim_{n\to\infty}\sum_{l=1}^{\left\lfloor \frac{n}{p}\right\rfloor }\mathbf{1}_{\left\{ lp\le n\right\} }\left(\frac{1}{n}+\varepsilon_{lp,n}\right)\\
 & =\lim_{n\to\infty}\frac{\left\lfloor \frac{n}{p}\right\rfloor }{n}+\lim_{n\to\infty}\sum_{l=1}^{\left\lfloor \frac{n}{p}\right\rfloor }\varepsilon_{lp,n}\\
 & \overset{\left(\ref{eq:2}\right)}{=}\frac{1}{p}.
\end{align*}

\noindent If $q\not=p$ is another prime, then, similarly, 

\noindent 
\begin{align*}
\mathbb{P}_{\infty}\left(A_{p}\cap A_{q}\right) & =\frac{1}{pq}+\sum_{l\ge1}\varepsilon_{pql}\\
 & \overset{\left(\ref{eq:2}\right)}{=}\frac{1}{pq}\\
 & =\mathbb{P}_{\infty}\left(A_{p}\right)\mathbb{P}_{\infty}\left(A_{q}\right).
\end{align*}

\noindent Therefore, the events $A_{p}$ and $A_{q}$ are independent.
In general, $\left(\ref{eq:2}\right)$ ensures that, for any positive
integer $k$, the events $A_{p_{1}},\ldots,A_{p_{l}}$ are independent
for $1<l\le k$.

Let $\delta_{p}\left(n\right)=\mathbf{1}_{p\vert n}$ so that 
\begin{align*}
\omega\left(n\right) & =\sum_{p\le n}\delta_{p}\left(n\right)
\end{align*}
is the number of distinct prime factors of $n$. The indicator variables
$\delta_{p}$ behave like Bernoulli variables $X_{p}$ that are independent
and identically distributed with 
\begin{align*}
\mathbb{P}\left(X_{p}=1\right) & =\frac{1}{p},\\
\mathbb{P}\left(X_{p}=0\right) & =1-1/p.
\end{align*}
The mean of $\sum_{p\le n}X_{p}$ is 
\begin{align*}
\mathbb{E}\left(\sum_{p\le n}X_{p}\right) & =\sum_{p\le n}\mathbb{P}\left(X_{p}=1\right)\\
 & =\sum_{p\le n}\frac{1}{p},
\end{align*}
and the variance of $\sum_{p\le n}X_{p}$ is
\begin{align*}
\text{Var}\left(\sum_{p\le n}X_{p}\right)= & \sum_{p\le n}\text{Var}\left(X_{p}\right)\\
= & \sum_{p\le n}\left(\mathbb{E}\left(X_{p}^{2}\right)-\left(\mathbb{E}\left(X_{p}\right)\right)^{2}\right)\\
= & \sum_{p\le n}\mathbb{E}\left(X_{p}^{2}\right)-\sum_{p\le n}\left(\left(\mathbb{E}\left(X_{p}\right)\right)^{2}\right)\\
= & \sum_{p\le n}\mathbb{E}\left(X_{p}\right)-\sum_{p\le n}\left(\left(\frac{1}{p}\right)^{2}\right)\\
= & \sum_{p\le n}\frac{1}{p}-\sum_{p\le n}\frac{1}{p^{2}}.
\end{align*}
Note that $\sum_{p\le n}\frac{1}{p}=\log\log n+O\left(1\right)$ by
Merten's Second Theorem (p. 22 of \cite{key-8}); the series $\sum_{p\le n}\frac{1}{p^{2}}$
converges (by comparison with $\sum_{1\le n}\frac{1}{n^{2}}$). Therefore,
the mean and variance are 
\begin{align*}
\sum_{p\le n}\frac{1}{p}= & \log\log n+O\left(1\right),\\
\sum_{p\le n}\frac{1}{p}-\sum_{p\le n}\frac{1}{p^{2}}= & \log\log n+O\left(1\right).
\end{align*}
This concludes the analogue of Kac's heuristic for $\mathbb{P}_{n}$
and justifies the parameters of the normal distribution appearing
in Theorem 2.

\section{Proving Theorem 2}

Define $\alpha_{n}\coloneqq n^{1/\log\log n}$.
\begin{lem}
As $n\to \infty$
\[
\left(\sum_{\alpha_{n}<p\le n}\left(\frac{1}{p}+\sum_{l=1}^{\left\lfloor \frac{n}{p}\right\rfloor }\varepsilon_{lp,n}\right)\right)/\left(\log\log n\right)^{1/2}\to0.
\]
\end{lem}

\begin{proof}
\noindent Given $n$ and any prime $p$ with $p>\alpha_{n}$, 
\[
-\frac{\left\lfloor \frac{n}{p}\right\rfloor }{n}\overset{\left(\ref{eq:4}\right)}{\le}\sum_{l=1}^{\left\lfloor \frac{n}{p}\right\rfloor }\varepsilon_{lp,n}\overset{\left(\ref{eq:6}\right)}{\le}\frac{D}{p}.
\]
 Therefore,
\begin{equation}
\frac{1}{p}+\sum_{l=1}^{\left\lfloor \frac{n}{p}\right\rfloor }\varepsilon_{lp,n}\in\left[0,\frac{D+1}{p}\right]\label{eq:7}
\end{equation}
for all $n$. Thus,
\[
\left(\sum_{\alpha_{n}<p\le n}\left(\frac{1}{p}+\sum_{l=1}^{\left\lfloor \frac{n}{p}\right\rfloor }\varepsilon_{lp,n}\right)\right)/\left(\log\log n\right)^{1/2}\to0
\]
due to $\left(\ref{eq:7}\right)$ along with the fact that Durrett (p.
135 of \cite{key-1}) shows 
\[
\left(\sum_{\alpha_{n}<p\le n}\frac{1}{p}\right)/\left(\log\log n\right)^{1/2}\to0.
\]
This proves Lemma 3.
\end{proof}
The following lemma is proved by Durrett (p. 156 of \cite{key-1}).
\begin{lem}
If $\varepsilon>0$, then $\alpha_{n}\le n^{\varepsilon}$ for large
$n$ and hence 
\begin{equation}
\frac{\alpha_{n}^{r}}{n}\to0\label{eq:8}
\end{equation}
for all $r<\infty$.
\end{lem}

\begin{proof}[Proof of Theorem 2]
\noindent Let $g_{n}\left(m\right)=\sum_{p\le\alpha_{n}}\delta_{p}\left(m\right)$
and let $\mathbb{E}_{n}$ denote expectation with respect to $\mathbb{P}_{n}^{*}$.
Then

\begin{align*}
\mathbb{E}_{n}\left(\sum_{\alpha_{n}<p\le n}\delta_{p}\right) & =\sum_{\alpha_{n}<p\le n}\mathbb{P}_{n}^{*}\left(m:\delta_{p}\left(m\right)=1\right)\\
 & =\sum_{\alpha_{n}<p\le n}\mathbb{P}_{n}^{*}\left(m:m=p,2p,\ldots,\left\lfloor \frac{n}{p}\right\rfloor p\right)\\
 & =\sum_{\alpha_{n}<p\le n}\sum_{l=1}^{\left\lfloor \frac{n}{p}\right\rfloor }\mathbb{P}_{n}^{*}\left(m:m=lp\right)\\
 & \overset{\left(\ref{eq:1}\right)}{=}\sum_{\alpha_{n}<p\le n}\sum_{l=1}^{\left\lfloor \frac{n}{p}\right\rfloor }\left(\frac{1}{n}+\varepsilon_{lp,n}\right)\\
 & =\sum_{\alpha_{n}<p\le n}\sum_{l=1}^{\left\lfloor \frac{n}{p}\right\rfloor }\frac{1}{n}+\sum_{\alpha_{n}<p\le n}\sum_{l=0}^{\left\lfloor \frac{n}{p}\right\rfloor }\varepsilon_{lp,n}\\
 & =\sum_{\alpha_{n}<p\le n}\frac{\left\lfloor \frac{n}{p}\right\rfloor }{n}+\sum_{\alpha_{n}<p\le n}\sum_{l=1}^{\left\lfloor \frac{n}{p}\right\rfloor }\varepsilon_{lp,n}\\
 & \le\sum_{\alpha_{n}<p\le n}\frac{1}{p}+\sum_{\alpha_{n}<p\le n}\sum_{l=1}^{\left\lfloor \frac{n}{p}\right\rfloor }\varepsilon_{lp,n},
\end{align*}
so by Lemma 3 it suffices to prove Theorem 2 for $g_{n}$; i.e.,
replacing $\omega\left(m\right)$ with $g_{n}\left(m\right)$
in the statement of Theorem 2 does not affect the limiting distribution. 

Let
\begin{eqnarray*}
S_{n} & \coloneqq & \sum_{p\le\alpha_{n}}X_{p},\\
b_{n} & \coloneqq & \mathbb{E}\left(S_{n}\right),\\
a_{n}^{2} & \coloneqq & \text{Var}\left(S_{n}\right).
\end{eqnarray*}
By Lemma 3, $b_{n}$ and $a_{n}^{2}$ are both $\log\log n+o\left(\left(\log\log n\right)^{1/2}\right)$,
so it suffices to show
\[
\mathbb{P}_{n}^{*}\left(m:g_{n}\left(m\right)-b_{n}\le xa_{n}\right)\to\mathbb{P}\left(Z\le x\right).
\]
An application of Theorem 3.4.5 of \cite{key-1} shows $\left(S_{n}-b_{n}\right)/a_{n}\to Z$,
and since $\left|X_{p}\right|\le1$, it follows from Durrett's second
proof of Theorem 3.4.5 \cite{key-1} that $\mathbb{E}\left(\left(S_{n}-b_{n}\right)/a_{n}\right)^{r}\to\mathbb{E}\left(Z^{r}\right)$
for all $r$. Using the notation from that proof (and replacing $i_{j}$
by $p_{j}$) it follows that
\begin{align*}
\mathbb{E}\left(S_{n}^{r}\right) & =\sum_{k=1}^{r}\sum_{r_{i}}\frac{r!}{r_{1}!\cdots r_{k}!}\frac{1}{k!}\sum_{p_{j}}\mathbb{E}\left(X_{p_{1}}^{r_{1}}\cdots X_{p_{k}}^{r_{k}}\right),
\end{align*}
where the sum $\sum_{r_{i}}$ extends over all $k$-tuples of positive
integers for which $r_{1}+\cdots+r_{k}=r$, and $\sum_{p_{j}}$ extends
over all $k$-tuples of distinct primes in $\left[n\right]$. Since
$X_{p}\in\left\{ 0,1\right\} $, the summand in $\sum_{p_{j}}\mathbb{E}\left(X_{p_{1}}^{r_{1}}\cdots X_{p_{k}}^{r_{k}}\right)$
is
\[
\mathbb{E}\left(X_{p_{1}}\cdots X_{p_{k}}\right)=\frac{1}{p_{1}\cdots p_{k}}
\]
by independence of the $X_{p}$'s. Moreover,
\begin{align*}
\mathbb{E}_{n}\left(\delta_{p_{1}}\cdots\delta_{p_{k}}\right) & \le\mathbb{P}_{n}\left(m:\delta_{p_{1}}\left(m\right)=\delta_{p_{2}}\left(m\right)\cdots=\delta_{p_{k}}\left(m\right)=1\right)\\
 & =\mathbb{P}_{n}\left(m:m=p_{1}\cdots p_{k},2p_{1}\ldots p_{k},\ldots,\left\lfloor \frac{n}{p_{1}\cdots p_{k}}\right\rfloor p_{1}\cdots p_{k}\right)\\
 & =\sum_{l=1}^{\left\lfloor \frac{n}{p_{1}\cdots p_{k}}\right\rfloor }\mathbb{P}_{n}\left(m:m=lp_{1}\cdots p_{k}\right)\\
 & \overset{\left(1\right)}{=}\sum_{l=1}^{\left\lfloor \frac{n}{p_{1}\cdots p_{k}}\right\rfloor }\left(\frac{1}{n}+\varepsilon_{lp_{1}\cdots p_{k},n}\right)\\
 & =\frac{\left\lfloor \frac{n}{p_{1}\cdots p_{k}}\right\rfloor }{n}+\sum_{l=1}^{\left\lfloor \frac{n}{p_{1}\cdots p_{k}}\right\rfloor }\varepsilon_{lp_{1}\cdots p_{k},n}.
\end{align*}
The two moments differ by at most 
\[
\text{max}\left\{ \frac{1}{p_{1}\cdots p_{k}}-\frac{\left\lfloor \frac{n}{p_{1}\cdots p_{k}}\right\rfloor }{n}-\sum_{l=1}^{\left\lfloor \frac{n}{p_{1}\cdots p_{k}}\right\rfloor }\varepsilon_{lp_{1}\cdots p_{k},n},\frac{\left\lfloor \frac{n}{p_{1}\cdots p_{k}}\right\rfloor }{n}+\sum_{l=1}^{\left\lfloor \frac{n}{p_{1}\cdots p_{k}}\right\rfloor }\varepsilon_{lp_{1}\cdots p_{k},n}-\frac{1}{p_{1}\cdots p_{k}}\right\} .
\]
Further, 
\begin{align*}
\frac{1}{p_{1}\cdots p_{k}}-\frac{\left\lfloor \frac{n}{p_{1}\cdots p_{k}}\right\rfloor }{n}-\sum_{l=1}^{\left\lfloor \frac{n}{p_{1}\cdots p_{k}}\right\rfloor }\varepsilon_{lp_{1}\cdots p_{k},n} & \le\frac{1}{p_{1}\cdots p_{k}}-\frac{\frac{n}{p_{1}\cdots p_{k}}-1}{n}-\sum_{l=1}^{\left\lfloor \frac{n}{p_{1}\cdots p_{k}}\right\rfloor }\varepsilon_{lp_{1}\cdots p_{k},n}\\
 & =\frac{1}{n}-\sum_{l=1}^{\left\lfloor \frac{n}{p_{1}\cdots p_{k}}\right\rfloor }\varepsilon_{lp_{1}\cdots p_{k},n},
\end{align*}
and 
\begin{align*}
\frac{\left\lfloor \frac{n}{p_{1}\cdots p_{k}}\right\rfloor }{n}+\sum_{l=1}^{\left\lfloor \frac{n}{p_{1}\cdots p_{k}}\right\rfloor }\varepsilon_{lp_{1}\cdots p_{k},n}-\frac{1}{p_{1}\cdots p_{k}} & \le\sum_{l=1}^{\left\lfloor \frac{n}{p_{1}\cdots p_{k}}\right\rfloor }\varepsilon_{lp_{1}\cdots p_{k},n}.
\end{align*}
Thus, the maximum becomes
\[
\text{max}\left\{ \frac{1}{n}-\sum_{l=1}^{\left\lfloor \frac{n}{p_{1}\cdots p_{k}}\right\rfloor }\varepsilon_{lp_{1}\cdots p_{k},n},\sum_{l=1}^{\left\lfloor \frac{n}{p_{1}\cdots p_{k}}\right\rfloor }\varepsilon_{lp_{1}\cdots p_{k},n}\right\} .
\]
The maximum equals $\frac{1}{n}-\sum_{l=1}^{\left\lfloor \frac{n}{p_{1}\cdots p_{k}}\right\rfloor }\varepsilon_{lp_{1}\cdots p_{k},n}$
if and only if $\sum_{l=1}^{\left\lfloor \frac{n}{p_{1}\cdots p_{k}}\right\rfloor }\varepsilon_{lp_{1}\cdots p_{k},n}\le\frac{1}{n}-\sum_{l=1}^{\left\lfloor \frac{n}{p_{1}\cdots p_{k}}\right\rfloor }\varepsilon_{lp_{1}\cdots p_{k},n}$,
and the latter is equivalent to $\sum_{l=1}^{\left\lfloor \frac{n}{p_{1}\cdots p_{k}}\right\rfloor }\varepsilon_{lp_{1}\cdots p_{k},n}\le\frac{1/2}{n}$.
On the other hand, if the maximum equals $\sum_{l=1}^{\left\lfloor \frac{n}{p_{1}\cdots p_{k}}\right\rfloor }\varepsilon_{lp_{1}\cdots p_{k},n},$
then inequality $\left(\ref{eq:5}\right)$ implies $\sum_{l=1}^{\left\lfloor \frac{n}{p_{1}\cdots p_{k}}\right\rfloor }\varepsilon_{lp_{1}\cdots p_{k},n}\le\frac{C}{n}$.
Therefore, the two $r$th moments differ by

{\small{}
\begin{eqnarray*}
\left|\mathbb{E}\left(S_{n}^{r}\right)-\mathbb{E}_{n}\left(g_{n}^{r}\right)\right| & \le & \sum_{k=1}^{r}\sum_{r_{i}}\frac{r!}{r_{1}!\cdots r_{k}!}\frac{1}{k!}\sum_{p_{j}}\left(\text{max}\left\{ \frac{1}{n}-\sum_{l=1}^{\left\lfloor \frac{n}{p_{1}\cdots p_{k}}\right\rfloor }\varepsilon_{lp_{1}\cdots p_{k},n},\sum_{l=1}^{\left\lfloor \frac{n}{p_{1}\cdots p_{k}}\right\rfloor }\varepsilon_{lp_{1}\cdots p_{k},n}\right\} \right)\\
 & \le & \sum_{k=1}^{r}\sum_{r_{i}}\frac{r!}{r_{1}!\cdots r_{k}!}\frac{1}{k!}\sum_{p_{j}}\frac{\max\left\{ C,\frac{1}{2}\right\} }{n}\\
 & \le & \frac{\max\left\{ C,\frac{1}{2}\right\} }{n}\left(\sum_{p\le\alpha_{n}}1\right)^{r}\\
 & \le & \max\left\{ C,\frac{1}{2}\right\} \frac{\alpha_{n}^{r}}{n}\\
 & \overset{\left(8\right)}{\to} & 0.
\end{eqnarray*}
}{\small\par}
This completes the proof of Theorem 2.
\end{proof}

\section{Illustrative Examples}
In this section, examples are given to help give a description of the class of distributions $\mathbb{P}^{*}_n$. It will be shown that the statement of Theorem 2 holds when $\mathbb{P}_{n}^{*}$ is replaced with either the Harmonic$\left(n\right)$ or Zipf$(n,s)$ distributions as long as $s\to 1$ for the latter.
\begin{example}
Given $n\in\mathbb{N},$ let $Q_{n}$ denote the Harmonic$\left(n\right)$
distribution on $\left[n\right]$ so that 
\[
Q_{n}\left(i\right)=\mathbf{1}_{\left\{ i\in\left[n\right]\right\} }\frac{1}{i\sum_{i=1}^{n}\frac{1}{i}}.
\]
In order to apply Theorem 2 to $Q_{n}$, it will be shown that the terms
$\varepsilon_{i,n}$ corresponding to $\mathbb{P}_{n}=Q_{n}$ satisfy
$\left(\ref{eq:2}\right)$, $\left(\ref{eq:5}\right)$, and $\left(\ref{eq:6}\right)$.
If $i\in\left[n\right]$, then $\left(\ref{eq:1}\right)$ implies
$\varepsilon_{i,n}=\frac{1}{i\sum_{i=1}^{n}\frac{1}{i}}-\frac{1}{n}$.
The left hand side of $\left(\ref{eq:2}\right)$ is 
\begin{align*}
\lim_{n\to\infty}\sum_{l=1}^{\left\lfloor \frac{n}{p_{1}\cdots p_{k}}\right\rfloor }\varepsilon_{lp_{1}\cdots p_{k},n} & =\lim_{n\to\infty}\sum_{l=1}^{\left\lfloor \frac{n}{p_{1}\cdots p_{k}}\right\rfloor }\left(\frac{1}{lp_{1}\cdots p_{k}\sum_{i=1}^{n}\frac{1}{i}}-\frac{1}{n}\right)\\
 & =\lim_{n\to\infty}\left(\frac{\sum_{l=1}^{\left\lfloor \frac{n}{p_{1}\cdots p_{k}}\right\rfloor }\frac{1}{l}}{p_{1}\cdots p_{k}\sum_{i=1}^{n}\frac{1}{i}}-\frac{\left\lfloor \frac{n}{p_{1}\cdots p_{k}}\right\rfloor }{n}\right)\\
 & =\frac{1}{p_{1}\cdots p_{k}}-\frac{1}{p_{1}\cdots p_{k}}\\
 & =0.
\end{align*}
The left hand side of $\left(\ref{eq:5}\right)$ becomes 
\begin{align*}
\sum_{l=1}^{\left\lfloor \frac{n}{p}\right\rfloor }\varepsilon_{lp,n} & =\sum_{l=1}^{\left\lfloor \frac{n}{p}\right\rfloor }\left(\frac{1}{lp\sum_{i=1}^{n}\frac{1}{i}}-\frac{1}{n}\right)\\
 & =\frac{\sum_{i=1}^{\left\lfloor \frac{n}{p}\right\rfloor }\frac{1}{i}}{p\sum_{l=1}^{n}\frac{1}{i}}-\frac{\left\lfloor \frac{n}{p}\right\rfloor }{n}\\
 & \le\frac{1}{p}-\left(\frac{\frac{n}{p}-1}{n}\right)\\
 & \le\frac{1}{n}\\
 & \le\frac{1}{p},
\end{align*}
so $\left(\ref{eq:5}\right)$ holds with $C=1$. Further, the
left hand side of $\left(\ref{eq:6}\right)$ becomes
\begin{align*}
\sum_{l=1}^{\left\lfloor \frac{n}{p_{1}\cdots p_{k}}\right\rfloor }\varepsilon_{lp_{1}\cdots p_{k},n} & =\sum_{l=1}^{\left\lfloor \frac{n}{p_{1}\cdots p_{k}}\right\rfloor }\left(\frac{1}{lp_{1}\cdots p_{k}\sum_{i=1}^{n}\frac{1}{i}}-\frac{1}{n}\right)\\
 & =\frac{\sum_{i=1}^{\left\lfloor \frac{n}{p_{1}\cdots p_{k}}\right\rfloor }\frac{1}{i}}{p_{1}\cdots p_{k}\sum_{i=1}^{n}\frac{1}{i}}-\frac{\left\lfloor \frac{n}{p_{1}\cdots p_{k}}\right\rfloor }{n}\\
 & \le\frac{1}{p_{1}\cdots p_{k}}-\left(\frac{\frac{n}{p_{1}\cdots p_{k}}-1}{n}\right)\\
 & =\frac{1}{n},
\end{align*}
so $\left(\ref{eq:6}\right)$ holds with $D=1$. This proves the
analogue of the Erd\H{o}s-Kac Theorem in the case of an integer-valued
random variable with the Harmonic$\left(n\right)$ distribution by
Theorem 2. 

\noindent \hfill{}$\square$
\end{example}

\begin{example}
Given $s>1$, denote by $Z_{s}$ the Zeta$\left(s\right)$ distribution
so that for any $j\in\mathbb{N}$, $Z_{s}\left(j\right)=\frac{1}{j^{s}\zeta\left(s\right)}$,
where $\zeta\left(j\right)=\sum_{j\ge1}\frac{1}{j^{s}}$ denotes the
Riemann zeta function. Since Theorem 2 involves distributions defined
on $\left[n\right]$, restrict the Zeta$\left(s\right)$ distribution
to $\left[n\right]$ and then normalize by dividing by $\sum_{i=1}^{n}\frac{1}{i^{s}\zeta\left(s\right)}$.
I.e., for $j\in\left[n\right]$,
\begin{align*}
Z_{s,n}\left(j\right) & \coloneqq\frac{\frac{1}{j^{s}\zeta\left(s\right)}}{\sum_{i=1}^{n}\frac{1}{i^{s}\zeta\left(s\right)}}\\
 & =\frac{1}{j^{s}\sum_{i=1}^{n}\frac{1}{i^{s}}};
\end{align*}
and $Z_{s,n}$ is known as the Zipf distribution with parameters $n$
and $s$. As $s\to1$,
the Zipf$\left(n,s\right)$ distribution $Z_{s,n}$ converges in distribution
to the Harmonic$\left(n\right)$ distribution $Q_{n}$. By Example 5, this ensures
that constraints $\left(\ref{eq:2}\right)$, $\left(\ref{eq:5}\right)$,
and $\left(\ref{eq:6}\right)$ will hold for $Z_{s,n}$ for $s$ sufficiently
close to $1$ and dependent on $n$. Therefore, the number of distinct
prime factors of a Zipf($s,n$) distribution tends to $\mathcal{N}\left(\log\log n,\log\log n\right)$
as $n\to\infty$ and $s\to1$.

\noindent \hfill{}$\square$
\end{example}

\section{Conclusion}
Theorem 2 generalizes the Erd\H{o}s-Kac Theorem to a larger family of distributions beyond the uniform distribution, and this theorem was proved by imposing the constraints $\left(\ref{eq:2}\right)$, $\left(\ref{eq:5}\right)$, and $\left(\ref{eq:6}\right)$ on $\mathbb{P}_n$. There are several ways to strengthen Theorem 2. As demonstrated in $\S$2, equation $\left(\ref{eq:2}\right)$ is equivalent to the independence
of the events $p_{1}\mathbb{N},\ldots,p_{k}\mathbb{N}$, with respect to $\lim_{n\to \infty}\mathbb{P}_{n}$, for distinct primes $p_{1},\ldots,p_{k}$. Moreover, independence naturally leads to a Gaussian law for the limiting distribution. Therefore, by weakening $\left(\ref{eq:2}\right)$, the limiting distribution will either not be a Gaussian or will consists of parameters other than $\mu=\sigma ^2=\log \log n$ in case it is a Gaussian. Example 5 provides a technique for which Theorem 2 can be applied to a particular
distribution on $\left[n\right]$ in order to conclude that the distribution of $\omega\left(\cdot\right)$ is asymptotically $\mathcal{N}\left(\log \log n, \log \log n\right)$. Providing further examples, or counterexamples, of familiar distributions will provide a better description of the class of distributions determined by $\left(\ref{eq:2}\right)$, $\left(\ref{eq:5}\right)$, and $\left(\ref{eq:6}\right)$.
Inequalities $\left(\ref{eq:5}\right)$ and $\left(\ref{eq:6}\right)$ were applied in order to apply the method of moments approach given by Durrett; the question remains as to the extent to which these inequalities can be weakened in order to still obtain sufficient bounds in the method of moments estimates.

Another way to generalize Theorem 2 would be to incorporate it with
other generalizations, e.g., \cite{key-4,key-5,key-6,key-7}. By incorporating Theorem 2 with these generalizations, further generalizations can be made in which the original setting is not $\left[n\right]$, the underlying distribution of the random-integer is not uniform, and $\omega\left(n\right)$ can be replaced with a more general function
$\omega\left(f\left(n\right)\right)$.

\end{document}